\documentclass[10pts]{amsart}
\usepackage{amsmath,amsfonts,amsthm,color}
\newtheorem{theorem}{Theorem}
\numberwithin{theorem}{section}
\newtheorem{corollary}[theorem]{Corollary}

\theoremstyle{definition}
\newtheorem{definition}[theorem]{Definition}

\numberwithin{equation}{section}

\newcommand\R{{\mathbb R}}
\newcommand\kk{\kappa}
\newcommand\g{{\mathfrak g}}
\newcommand\oa{{\mathfrak o}}

\newcommand\gl{{\mathfrak gl}}
\newcommand\s{\varsigma}
\newcommand\h{{\mathfrak h}}
\newcommand\m{{\mathfrak m}}
\newcommand\F{{\mathcal F}}

\newcommand\HH{{\mathcal H}}

\newcommand\C{{\mathcal C}}

\newcommand\Lo{{\mathcal L}}

\newcommand\M{{\mathcal M}}

\newcommand\var[1]{\frac{\delta #1} {\delta L}}

\newcommand\K{{\mathcal K}}

\newcommand\pr[2]{\langle #1,#2 \rangle_J}

\newcommand\GL{{\rm GL}}
\newcommand\SL{{\rm SL}}

\newcommand\bok{{\bf k}}

\newcommand\bu{{\bf u}}

\newcommand\rr{{\bf r}}
\newcommand\bs{{\bf s}}
\newcommand\hu{{\hat u}}

\begin{document}
\title[Completely integrable flow on the light cone]
      {A completely integrable flow of star-shaped curves on the light cone in Lorentzian $\R^4$}
\author{T. C. Anderson}
\address{Mathematics Department\\ Brown University\\
 Providence, RI 02912} 
\email{tcanderson@math.brown.edu}
\author{G. Mar\'\i~Beffa}
\thanks{{G. M. B. supported by NSF grant DMS \#0804541}, {T. C. A. supported by UW Madison Hilldale fellowship and the L\&S Summer Honors Senior Thesis grant}}
\address{Mathematics Department\\
         University of Wisconsin\\
         Madison, Wisconsin 53706} 
\email{maribeff@math.wisc.edu}

\date{July, 2010}
\maketitle
\centerline{{\small\it In memory of Ron Peters}}
\begin{abstract} In this paper we prove that the space of differential invariants for curves with arc-length parameter in the light cone of Lorentzian $\R^4$, invariants under the centro-affine action of the Lorentzian group, is Poisson equivalent to the space of conformal differential invariants for curves in the M\"obius sphere. We use this relation to find realizations of solutions of a complexly coupled system of KdV equations as flows of curves in the cone. 
\end{abstract}
\section{Introduction} Many - if not most - completely integrable systems posses what is commonly known as geometric realizations; that is, they can be viewed as curve flows in a geometric manifold inducing the integrable system in the curvatures (or differential invariants) of the flow. The best known example is that of the nonlinear Shr\"odinger equation
\[
\phi_t = i \phi_{xx} + \frac i 2|\!|\phi|\!|^2 \phi
\]
whose geometric realization as a 3D Euclidean flow is given by the Vortex-filament flow
\[
 u_t = k B
 \]
via the {\it Hasimoto transformation} $\phi = k e^{i\int \tau dx}$, where $k$ and $\tau$ are the curvature and torsion of the flow $u$ and $B$ is its binormal (see \cite{Ha}). Another well-known example is that of the KdV equation
\[
k_t = k_{xxx} + 3 k k_x
\]
whose realization in the projective line $\mathbb{RP}^1$ is known to be the Schwarzian KdV equation
\[
u_t = u_{xxx} -\frac 32 \frac{u_{xx}^2}{u_x} = u_x S(u)
\]
via the trasnformation $k = S(u) = \frac{u_{xxx}}{u_x} - \frac 32 \frac{u_{xx}^2}{{u_x}^2}$; the invariant $k$ is the Schwarzian derivative of the flow - the only projective differential invariant on the line. One can find a large number of geometric realizations of many different integrable systems in the literature, see for example, \cite{A}, \cite{GR}, \cite{KQ1,KQ2}, \cite{M1}, \cite{TT}, \cite{TU1} and references within.

The second author of this paper described a closer relationship between integrable systems and geometric realizations in \cite{M1}. Indeed, if the geometric manifold is a homogeneous manifold of the form $G/H$, where $G$ is semisimple, the space of differential invariants is naturally endowed with a Hamiltonian structure, one she called the {\it geometric Poisson bracket} for curve flows. This bracket is obtained via the reduction of a well-known bracket defined on the manifold of loops in the dual of the Lie algebra $\g$, to a quotient of this manifold. A second compatible bracket will not reduce in general, and its reduction indicates the existence of a geometric realization of an associated integrable system, as explained in \cite{M1}. Both the nonlinear Schr\"odinger equation and the KdV equation are biHamiltonian systems and their structures are obtained when reducing these two general brackets to the space of Euclidean and projective invariants, respectively. Other such cases include modified KdV, Sine Gordon, Sawada-Koterra, Gelfan'd-Dikii flows, etc.

In \cite{CIM}, the authors showed that one can consider planar curves $\gamma$ as curves in $\mathbb{RP}^1$ (via a simple projectivization) insofar as the curve is star-shaped, that is $\gamma$ and $\gamma'$ remain independent and the curve never points towards the radial direction. Using this relation they proceeded to show that completely integrable systems with projective realizations (for example KdV) also have realizations as flows of star-shaped curves, invariant under the centro-affine action of $\SL(2)$. The realization of KdV as a flow of star-shaped curves was previously discovered by Pinkall in \cite{P}. The authors of \cite{CIM} proved that, if we assume that the planar curves are parametrized by centro-affine arc-length parameter, the space of differential invariants in both centro-affine plane and the projective line were Poisson equivalent when endowed by the geometric Poisson structures. This allowed them to establish planar geometric realizations by star-shaped flows not only for KdV, but for any projective Hamiltonian system, including, for example, the Sawada-Koterra equation. They also used centro-affine moving frames along the flow to find realizations of solitons solutions of KdV as star-shaped planar solitons.

In this paper we look at a related situation involving projectivization. It is well-known (see for example \cite{GS}) that the conformal M\"obius sphere $S^2$ is geometrically equivalent to the projectivization  of the light cone in Lorentzian $\R^4$, both homogeneous spaces. Furthermore, a familiar completely integrable system - the complexly coupled system of KdV equations - is known to have a conformally invariant realization in the M\"obius sphere (see \cite{M3}). Thus, it is natural to ask whether or not the coupled system has a realization as flows in the cone, invariant under Lorentzian transformations. 

In section 3 we classify differential invariants for curves in the cone under the centro-affine Lorentzian action, and we find a group-based moving frame along them. Using our moving frame, we write a formula for a general evolution of curves in the cone and identify those that preserve arc-length. We then find the Lorentzian Poisson structure explicitly. In section 4 we review the same study for conformal curves (the study was done in \cite{M2} for the general conformal case), and we match it to the Lorentzian case. In section 5 we prove that, if  curves in the cone are parametrized by Lorentzian centro-affine arc-length, then projectivization takes differential invariants in the cone to differential invariants in the sphere in a 1-to-1 fashion. It also takes evolutions in the cone that are invariant under Lorentzian transformations and preserve arc-length to conformally invariant evolutions of curves in $S^2$. And vice versa, any projective curve in the sphere can be uniquely lifted to a curve in the cone, parametrized by arc-length and with equal invariants; and any projective evolution can be lifted to an evolution in the cone that preserves arc-length.

In section 6 we finally show that if we assume that the curves in the cone are parametrized by arc-length, the space of invariants in both the cone and the sphere are Poisson equivalent. In fact, the Poisson bracket defined on Lorentzian invariants foliates into Poisson submanifolds according to the value of the arc-length invariant $k_0$. If we fix the value $k_0 =1$, then the resulting Poisson structure is identical to that of the conformal invariants. This relation will allow us to find geometric realizations in both sphere and cone for any system that is Hamiltonian with respect to the geometric Poisson structure, in particular that of the complexly coupled KdV system. To conclude the paper we provide the Lorentzian realization of the coupled KdV system and describe how to lift soliton solutions to the cone.

\section{Notation, definitions and background}

\subsection{Moving frames, differential invariants and invariant flows}
The classical concept of moving frame
was developed by \'Elie Cartan (\cite{C1}). A classical moving frame along a curve
in a manifold $M$  is a curve in the frame bundle of the manifold over the curve, invariant under
the action of the transformation group under consideration. This method is a very powerful tool, but its 
explicit application relied on intuitive choices that were not clear in a general setting.
Some ideas in Cartan's work and later work of Griffiths (\cite{Gri}), Green (\cite{Gre}) and others laid the
 foundation for the concept of a group-based moving frame, that is, an equivariant map between 
 the jet space of curves in the manifold and the group of transformations. Recent work by 
 Fels and Olver (\cite{FO1, FO2}) finally gave the precise definition of the group-based moving
 frame and extended its application beyond its original geometric picture to a 
 large number of applications.  In this section we will describe Fels and Olver's moving frame
and its relation to the classical moving frame. We will also introduce
 some definitions that are useful for the study of  Poisson brackets  and
biHamiltonian nonlinear PDEs. From now on we will assume $M = G/H$ with $G$
 acting on $M$ via left multiplication on representatives of a class. We will also assume that curves in $M$ are {\it parametrized}
 and, therefore, the group $G$ does not act on the parameter.
 
 \begin{definition}{\rm 
Let $J^k(\R,M)$ the space of $k$-jets of  curves, that is,  the set of equivalence classes of curves in 
$M$ up to $k^{\rm th}$ order of contact. If we denote by $u(x)$ a curve in $M$, the jet space has local 
coordinates that can be represented by $u^{{\bf (k)}}= (x, u, u', u'', \dots, u^{(k)})$. 
The group $G$ acts naturally on parametrized curves, therefore it acts naturally
on the jet space via the formula
\[
g\cdot u^{{\bf (k)}} = (x, g\cdot u, (g\cdot u)', (g\cdot u)'', \dots )
\]
where by $(g\cdot u)^{(k)}$ we mean the formula obtained when one differentiates $k$ times $g\cdot u$ and 
then writes the result in terms of
$g$, $u$, $u'$, etc. This is usually called the {\it prolonged} action of $G$ on $J^k(\R,M)$.}
 \end{definition}
\begin{definition} {\rm A function 
\[
I: J^k(\R,M) \to \R
\]
is called a $k$th order {\it differential invariant} if it is invariant with respect to the prolonged action of $G$.}
\end{definition}
\begin{definition}{\rm 
A map 
\[
\rho: J^k(\R,M) \to G
\]
is called a left (resp. right)  {\it moving frame} if it is equivariant with respect to the prolonged action of $G$
on $J^k(\R,M)$ and the left action (resp. right inverse action) of $G$ on itself. }
\end{definition}
If a group acts (locally) {\it effectively on subsets}, then for $k$ large enough the prolonged action is locally free on regular jets. This guarantees the existence of a moving frame on a neighborhood of a regular jet (for example, on a neighborhood of a generic curve, see \cite{FO1, FO2}). 

 The group-based moving frame already appears in a familiar method for calculating the curvature of a curve $u(s)$ in the 
Euclidean plane. In this method one uses a translation to take $u(s)$ to the origin, and a rotation to make one of the axes tangent to the curve. The curvature can classically be found as the 
 coefficient of the second order term in the expansion of the curve around $u(s)$. The crucial observation made by Fels and Olver is that {\it the element of the group} carrying out the translation and rotation depends on $u$ and its derivatives and so it defines a map from the jet space to the group. {\it This map is  a right moving frame},
 and it carries all the geometric information of the curve. In fact, Fels and Olver developed 
 a similar normalization process to find right moving frames (see \cite{FO1, FO2} and our next  Theorem).
\begin{theorem} \label{norm}(\cite{FO1, FO2}) Let $\cdot$ denote the prolonged
action of the group on $u^{(k)}$ and assume we have {\it normalization equations} of the form
\[
g \cdot u^{(k)} = c_k
\]
where $c_k$ are constants (they are 
called \emph{normalization constants}). 
Assume we have enough normalization equations so as to determine $g$ as a function of $u, u', \dots$. Then $g = \rho$ is a \emph{right moving frame}.
\end{theorem}
The direct relation between classical moving frames and group-based moving frames is stated 
in the following theorem.

\begin{theorem} \label{classical}(\cite{M2}) Let $\Phi_g: G/H \to G/H$ be defined by the action of  $g\in G$. That is
$\Phi_g([x]) = [g x]$. Let $\rho$ be a group-based left moving frame with $\rho \cdot o = u$
where $o = [H] \in G/H$. Let $e_i$, $i=1,\dots,n$ be generators of the vector space $T_o G/H$. Then, $T_i = d\Phi_\rho(o)e_i$ form a classical moving frame.
\end{theorem}

We will next describe the equivalent to the classical Serret-Frenet equations. This concept
is fundamental in our Poisson geometry study. 

\begin{definition}{\rm Consider $K dx$ to be the horizontal component of the pullback of the left (resp. right)
Maurer-Cartan form of the group $G$ via a group-based left (resp. right)  moving frame $\rho$. That is
\[
K = \rho^{-1} \rho_x \in \g \hskip4ex ({\rm resp.} \hskip 2ex K = \rho_x\rho^{-1})
\]
 and $K$ describes the first order differential equation satisfied by $\rho$. We call $K$ the left (resp. right) {\em Maurer-Cartan} element of the algebra (or Maurer-Cartan matrix if $G\subset \GL(n,\R)$), and $ \rho_x = \rho K$ the {\it left (resp. right) Serret-Frenet equations} for the moving frame $\rho$.}
\end{definition}
Notice that, if $\rho$ is a left moving frame, then $\rho^{-1}$ is a right moving frame and their Serret-Frenet equations are the negative of each other. A complete set of 
generating differential invariants
 can always be found among the
coefficients of group-based Serret-Frenet equations generated by normalization equations. The following
Theorem can be found in
 \cite{Hu}.

\begin{theorem}  Let $\rho$ be a (left or right) moving frame
 along a curve $u$. Then, the coefficients of the (left or right) Serret-Frenet equations
for $\rho$ contain a basis for the space of differential  invariants of the curve. That is, any other differential 
invariant for the curve is a function of the generators of $K$ (its entries if $G\subset \GL(n,\R)$) and their derivatives with respect to $x$.
\end{theorem}
If $\rho$ is obtained through normalizations, there are formulas relating $K$ directly to the invariantization of jet
coordinates. 
They are called recurrence formulas in \cite{FO1}, and the theorem below is the adaptation of their result to our particular case, and can be found in \cite{M1}.
 \begin{theorem}\label{Knorm} Assume a right invariant moving frame is determined by the normalization equations
\[
\left(g\cdot u^{(r)}\right)^\alpha = c_r^\alpha
\]
for some choices of $\alpha$ and $r$, where $c_r = (c_r^\alpha)$ ($\alpha$ indicates individual
coordinates). Let $K = \rho_x \rho^{-1}$ be the associated Maurer-Cartan element. Let $I_r^\alpha = \rho\cdot (u^\alpha)^{(r)}$ for any $r = 0,1,2, \dots$ and any
$\alpha = 1,\dots, \dim M$. Then $K$ is determined by the equations
\begin{equation}\label{Knor}
\left(K\cdot I_{r}\right)^\alpha = -I_{r+1}^\alpha + \left(I_r^\alpha\right)' ,
\end{equation}
where $K\cdot u^{(r)}$ denotes the prolonged infinitesimal action of the Lie algebra on $J^{r}(\R, \M)$.
\end{theorem}

Notice that $I_r^\alpha = c_r^\alpha$ for the normalized components, otherwise they are differential invariants.

If one has a group-based moving frame, then one also has a general formula for an evolution of curves in $G/H$, invariant under the action of $G$ (that is, so that $G$ takes solutions to solutions).
 
 \begin{theorem}(\cite{M1})\label{invev} Let $u(x,t)$ be a one parameter family of curves in $G/H$; let $\rho$ be a {\bf left} group-based moving frame and let $d\Phi_\rho(o) e_i = T_i$ be an associated classical moving frame. Then, any evolution of curves in $G/H$ invariant under the action of $G$ can be written as
 \begin{equation}\label{guev}
 u_t = r_1 T_1+\dots r_n T_n = d\Phi_\rho(o)\rr 
 \end{equation}
 where $\rr = (r_1,\dots, r_n)^T$ and $r_i$ are differential invariants, that is, functions of the entries of $K$ and their derivatives.
 \end{theorem}

Finally, if $\rho$ is a left moving frame along the flow $u(t,x)$, solution of (\ref{guev}), the entries $r_i$ determine the entries of $N = \rho^{-1}\rho_t$ in a very precise way, where $\rho_t$ is the evolution induced on $\rho$ by (\ref{guev}). 

\begin{theorem}\label{Nsplit} (\cite{M1}) Assume $M = G/H$ and let $\s:G/H\to G$ be a section, that is $\pi(\s(x)) = x$. Assume $\h\oplus\m$, where $\h$ is the subalgebra of $H$ and $m = \s_\ast(T_o G/H)$ is a complement to $\h$ as vector subspaces. Let $N = N_h+N_m$ according to the splitting of the algebra.

Then
\[
N_m = d\s(o)\rr, 
\]
where $\rr$ is given in (\ref{guev}).
\end{theorem}

\subsection{Geometric Hamiltonian structures}\label{geomst} Assume $\g$ is semisimple.
One can define two natural Poisson brackets on $\Lo\g^\ast$ (see \cite{PS} for more information);
 namely, if $\HH, \F: \Lo\g^\ast \to \R$ are two functionals defined on $\Lo\g^\ast$ and if
 $L \in \Lo\g^\ast$, we define
 \begin{equation}\label{br1}
 \{\HH, \F \}_1(L) = \int_{S^1} \langle B\left(\var{\HH} \right)_x + ad^\ast(\var{\HH}) (L) , \var \F\rangle dx
 \end{equation}
where $B$ is an invariant bilinear form that can be used to identify the algebra with its dual (usually the trace of the product), $\langle, \rangle$ is the natural coupling between $\g^\ast$ and $\g$ (usually the trace of the product if we identify $\g$ and $\g^\ast$), and where
$\var{\HH}$ is the variational derivative of $\HH$ at $L$ identified, as usual, with an element of 
$\Lo\g$. 

One also has a compatible family of second brackets, namely
\begin{equation}\label{br2}
\{\HH, \F\}_2(L) = \int_{S^1} \langle (ad^\ast(\var{\HH}) (L_0) , \var \F\rangle dx
\end{equation}
where $L_0\in\g^\ast$ is any constant element.
From now on we will identify $\g^\ast$ with $\g$ using the trace and we will also assume that our curves on homogeneous manifolds have a {\it group monodromy}, i.e., there exists $m\in G$ such that 
\[
u(t+T) = m\cdot u(t)
\]
where $T$ is the period. Under these assumptions, the differential invariants will be periodic.

The following theorem is the foundation of the definition of Geometric Hamiltonian structures. It was proved in \cite{M1}.

\begin{theorem}\label{Kred} Let $\rho$ be a left or right moving frame along a curve $u$, determined by normalization equations. Let $\K$ be the manifold of Maurer-Cartan matrices $K$ for nearby curves, generated using the same normalization equations. Then, $\K \cong U/\Lo H$ where $U\subset \Lo\g^\ast$ is an open set, and where $\Lo H$ acts on $U$ via a gauge transformation. Furthermore, the Poisson bracket defined on $\Lo\g^\ast$ by (\ref{br1}) is reducible to the submanifold $\K$. We call this first reduced Poisson bracket a Geometric Poisson bracket on $G/H$.
\end{theorem}
The reduction of these Poisson brackets can often  be found explicitly through algebraic manipulations. Indeed, if an extension $\HH$ of Hamiltonian functional $h:\K \to \R$ is constant on the gauge leaves of $\Lo H$, then its variational derivative will satisfy
\begin{equation}\label{Hcond}
\left(\frac{\delta \HH}{\delta L}(K)\right)_x + [K, \frac{\delta \HH}{\delta L}(K)] \in \h^0
\end{equation}
where $\h^0\subset \g^\ast$ is the annihilator of $\h$, and where $K$ is any Maurer-Cartan element. This relation is often sufficient to determine $\frac{\delta \HH}{\delta L}(K)$  completely and with it the reduced Poisson bracket; the reduced Poisson bracket will be defined through the application of (\ref{br1}) to two such extensions. 

The Poisson bracket (\ref{br2}) does not reduce in general to this quotient. When it does, it indicates the existence of an associated completely integrable system. The geometric Poisson bracket above is directly related to invariant evolutions through our next theorem. Assume
\begin{equation}\label{uinvev}
u_t = W(u, u', u'', \dots)
\end{equation}
is an evolution of curves invariant under the action of the group. Assume (\ref{uinvev}) induces an evolution of the form
\begin{equation}\label{kevg}
\bok_t = Q(\bok,\bok',\bok'',\dots)
\end{equation}
on a generating system ${\bf k} = (k_i)$ of differential invariants of the flow $u(t,x)$. We say that {\it (\ref{uinvev}) is a $G/H$-geometric realization of the flow (\ref{kevg})}.

Assume now, as before, that $\g = \h\oplus\mathfrak{m}$ and that $\s:G/H \to G$ is a section that identifies $T_oG/H$ with $\m$. Our following theorem finds geometric realizations for any Geometric Hamiltonian flow, Hamiltonian with respect to the reduced Poisson bracket. 

\begin{theorem}\label{greal} (\cite{M1}) Assume $\K$ is described by an affine subspace of $\Lo\g^\ast$. Let $h: \K \to \R$ be a Hamiltonian functional and $\HH:\Lo\g^\ast\to \R$ an extension of $h$, constant on the leaves of $\Lo H$ under the  gauge action.
Let $\frac{\delta \HH}{\delta L}(\bok) = \frac{\delta \HH}{\delta L}(\bok)_\m + \frac{\delta \HH}{\delta L}(\bok)_\h$ be the components of the variational derivative according to the splitting of the algebra.  Then  
\[
u_t = d\Phi_\rho(o)d\s(o)^{-1}\frac{\delta \HH}{\delta L}(\bok)_\m
\]
is a geometric realization of the reduced Hamiltonian system with Hamiltonian functional $h$. Notice that this evolution is of the form (\ref{guev}) with 
\begin{equation}\label{compatibility} \frac{\delta \HH}{\delta L}(\bok)_\m = d\s(o)\rr.
\end{equation}

\end{theorem}

Equation (\ref{compatibility}) is often referred to as the {\it compatibility condition}. 

\subsection{Definition of the light cone and the $2$-conformal sphere as homogeneous spaces}

\subsubsection*{The light cone in Lorentzian geometry} Let $(u_0,u_1,u_2,u_3)\in\mathbb{R}^4$; we define the inner product on $\R^4$ as the one associated to the Minkowski (degenerate) metric, $\pr{u}{u} = u\cdot u = \|u\|_J^2 = u^TJu,$ where \[J = \begin{pmatrix} 0&0&0&-1\\0&1&0&0\\0&0&1&0\\-1&0&0&0\\\end{pmatrix}.\]   This inner product defines the Lie group $O(3,1) = \{\Theta\in \GL(4) | \Theta^TJ\Theta = J\}$ as the elements of $\GL(4)$ that preserve it. With this representation, its Lie algebra is easily found to be defined by elements of $\gl(4)$ of the form 
\begin{equation}\label{alg}\begin{pmatrix}a&z^T&0\\w&B&z\\0&w^T&-a\end{pmatrix}\end{equation} 
where $z = (z_1,z_2)^T,$ $w = (w_1,w_2)^T,$ and where $B$ is a skew-symmetric matrix.

The light cone is defined as $C_L = \{u\in\R| \pr{u}{u} = 0\}$.
 In particular, if $u\in C_L$, $\hat{u} = \begin{pmatrix}u_1\\u_2\end{pmatrix},$ and we call $\|\hat{u}\| = \sqrt{u_1^2+u_2^2},$ then $u_1^2+u_2^2 = \|\hat{u}\|^2 = 2u_0u_3$.  The Lorentzian group  naturally acts on the light cone as
\begin{equation}\label{linear}
\begin{array}{ccc}\Theta(3,1)\times C_L&\to& C_L\\
(\Theta,u)&\to& \Theta u
\end{array}
\end{equation}
and this action is transitive. Define $H$ to be the isotropy subgroup of $e_4$. Then, the cone is isomorphic to the homogeneous space $O(3,1)/H$ and the isomorphism is given by $u \to [g_u]$ where $g_u e_4 = u$. Indeed, the natural action of $O(3,1)$ on $O(3,1)/H$ defined as $g\cdot[\hat g] = [g\hat g]$ is given by (\ref{linear}) under this isomorphism: according to $g g_u\cdot e_4 = g\cdot u$ the isomorphism takes $g u = g\cdot u$ to $[gg_u]= g\cdot [g_u]$.\vskip 2ex

\subsubsection*{The conformal sphere}
 The $2$-M\"{o}bius sphere, $M$, a flat model for the conformal plane, is obtained as the projectivization of the cone, and the Lorentzian action also projectivizes into conformal transformations (\cite{GS}).  The projectivization of the cone is topologically equivalent to a 2-sphere and we will locally represent it in coordinates $m=\begin{pmatrix}m_1\\m_2\end{pmatrix}$ as follows.   Define the lift from $M$ to $C_L$ as 
\begin{equation}
\label{projectmap}
\Lambda:M\to\C_L\hskip 5ex
\Lambda: m \to \begin{pmatrix}u_0\\m_1\\m_2\\1\end{pmatrix} = \hat{m},
\end{equation}
 where $u_0 = \frac{1}{2}\|m\|^2 = \frac{1}{2}(m_1^2+m_2^2)$.
This is the standard lift from $M$ to $C_L$ associated to the local projectivization $\Pi:C_L \to M$
$$\Pi:\begin{pmatrix}u_0\\u_1\\u_2\\u_3\end{pmatrix}\to\begin{pmatrix}\frac{u_1}{u_3}\\\frac{u_2}{u_3}\end{pmatrix}.$$

The Lorentzian group acts on $M$ by conformal transformations using $\Lambda$, namely  
$$\begin{array}{ccc}\Theta(3,1)\times M&\to& M\\
(\Theta,m)&\to&\Pi\Theta\Lambda(m) = \Theta\cdot m\end{array}.
$$
Locally, elements of $O(3,1)$ can be factored as
\begin{equation}
\label{decomposed}
g = g_1g_0g_{-1} = \begin{pmatrix}1&0&0\\\xi&I&0\\\frac{1}{2}\|\xi\|^2&\xi^T&1\end{pmatrix}
\begin{pmatrix}\alpha&0&0\\0&A&0\\0&0&\alpha ^{-1}\end{pmatrix}
\begin{pmatrix}1&v^T&\frac{1}{2}\|v\|^2\\0&I&v\\0&0&1\end{pmatrix}
\end{equation}
where $\xi = (\xi_1,\xi_2)^T, v = (v_1,v_2)^T$ and $A\in O(2)$.

With this notation, the general formula for the conformal action of the Lorentzian group is: 
\begin{equation}
\label{sphereact}
g\cdot m = \frac{
\alpha\xi(\frac{\|m\|^2}{2}+v^Tm+\frac{\|v\|^2}{2})+A(m+v)}
{\frac{\alpha}{2}\|\xi\|^2(\frac{\|m\|^2}{2}+v^Tm+\frac{\|v\|^2}{2})+\xi ^TA(m+v)+\alpha ^{-1}}.
\end{equation}


\section{The local geometry of non-radial curves in the light cone}
In this section we will find a group-based moving frame along curves in the cone, invariant under the linear action of $O(3,1)$, and we will use it to describe invariants, Serret-Frenet equations, geometric Poisson brackets and their geometric realizations. 
\subsection{Group based moving frame}
We will use the normalization process described in our previous section to construct the right group-based moving frame for the cone.  Consider the element $g\in G$ locally defined as in (\ref{decomposed})

\vskip 2ex
{\em Zeroth order normalization}. The first two matrices, $g_1$ and $g_0,$ in (\ref{decomposed}) are members of the isotropy subgroup of $e_4 = \begin{pmatrix}0\\{\bf 0}\\1\end{pmatrix}.$ Therefore, by taking $u$ to $e_4$ we can determine the remaining factor $g_{-1} = g_u^{-1}$. 
\begin{equation}
\label{c0}
g_{-1}\cdot u = g_{-1}\cdot\begin{pmatrix}u_0\\\hat{u}\\u_3\end{pmatrix} = c_0 = e_4 \end{equation}
  
   From here $v = -\frac{\hat{u}}{u_3}$ and $\alpha = u_3$.  Since $C_L$ consists of an upper cone and a lower one, we  will restrict to the upper cone $\|u'\|_J^2 = (u')^TJu'>0.$  However, you can cover the other case and get the same results, just making sure that you keep track of the negative signs. Notice that the non-degenerate condition $\|u'\|_J\ne 0$ implies that $u'$ is never in the radial direction of the cone.
   \begin{definition} We will say that a curve $u$ is non-radial or {\it star-shaped} if $u'$ is never in the radial direction of the cone; for example, if $\|u'\|_J \ne 0$.
   \end{definition}
   \vskip 2ex
{\em First order normalization}.
To find $A$ and $\xi_1,$ we need to normalize the first prolongation of the action. Since the action is linear, its recurrent prolongations are given by the matrix multiplication of $g$ by the corresponding derivative. The first normalization equation is thus given by
\begin{equation}
\label{c1}
gu' = c_1 = \begin{pmatrix}0\\\|u'\|_J\\0\\0\end{pmatrix}.
  \end{equation}
 We can normalize only three equations since the length of $u'$ is necessarily preserved by $g$.  After substituting the previously found values for $v$ and $\alpha$, 
and assuming $u_3\ne 0$, we find
   \[
   A^T\begin{pmatrix}1\\0\end{pmatrix} = \frac{\hat{u}' - \displaystyle\frac{\hat{u}}{u_3}u_3'}{\|u'\|_J}.
   \]
That is,
\begin{equation} \label{A}A = \frac{1}{u_3\|u'\|_J}\begin{pmatrix}\det\begin{pmatrix}u_3&u_1\\u_3'&u_1'\end{pmatrix}&
\det\begin{pmatrix}u_3&u_2\\u_3'&u_2'\end{pmatrix}\\ & \\
\det\begin{pmatrix}u_2&u_3\\u_2'&u_3'\end{pmatrix}&
\det\begin{pmatrix}u_3&u_1\\u_3'&u_1'\end{pmatrix}\end{pmatrix} = {(a_{ij})}.\end{equation}
We also find 
\begin{equation}\label{xi1}
\xi_1 = -\frac{u_3'}{u_3\|u'\|_J}.
\end{equation}

\vskip 2ex
{\it Second order normalization.}
Finally, we normalize the second prolonged action
 \begin{equation}
\label{c2}
gu'' = c_2 = \begin{pmatrix}a\\b\\c\\d\end{pmatrix}.
 \end{equation}
Only one of these entries can be normalized since three of the entries are already determined.  For example, $\pr{u'}{u''} = c_1^TJc_2 = \|u'\|_Jb,$ so \[
b = \frac{\pr{u'}{u''}}{\|u'\|_J} = (\|u'\|_J)'.\]  Also $\|u''\|^2_J = c_2^TJc_2 = -2ad +b^2+c^2$.  Using the fact that $\pr{u}{u}$ and its derivatives vanish, and substituting previously found values, one can directly check that  $a = \|u'\|_J^2$. Normalizing $c=0,$ we find \begin{equation}\label{xi2}
\xi_2 = \frac1{u_3\|u'\|_J^2 }\det\begin{pmatrix} u_1&u_2&u_3\\ u'_1&u'_2&u'_3\\
u_1''&u_2''&u_3''\end{pmatrix}
\end{equation}
and hence
 \[
 d = \frac{-\|u''\|_J^2\|u'\|_J^2 + \pr{u'}{u''}^2}{2\|u'\|_J^3}.
 \]  

We can now combine all of the components into the right moving frame:
\begin{equation}
\label{coneframe}
\rho_L = \begin{pmatrix}u_3&\hat{u}^T&u_0\|\hat{u}\|^2\\u_3\xi&u_3\xi\cdot v^T+A&u_0\xi\|\hat{u}\|^2+Av\\
\frac{u_3}{2}\|\xi\|^2
&\frac{u_3}{2}\|\xi\|^2v^T+\xi^TA&\frac{u_0}{2}\|\xi\|^2\|\hat{u}\|^2+\xi^TAv+u_3 ^{-1}\end{pmatrix}
\end{equation}
where $A$ and $\xi$ are as above. The left moving frame will be its inverse.
\subsection{Maurer Cartan matrix and differential invariants}
We are now in position to find the right Maurer-Cartan matrix $K_L = (\rho_L)_x\rho_L^{-1}$, and with it a generating system of invariants.

  \begin{theorem}
\begin{equation}\label{KL}K_L = \begin{pmatrix}0&-k_0&0&0\\ -k_1 &0&0&-k_0\\
-k_{2}&0&0&0\\0&-k_1&-k_2&0\end{pmatrix}
\end{equation} 
where 
\[
k_0 = \|u'\|_J, \hskip 2ex k_1 = \frac{-\|u''\|_J^2\|u'\|_J^2 + \pr{u'}{u''}^2}{2\|u'\|_J^4}
\]
and
\[
k_2 = \frac1{u_3\|u'\|_J^2}\det\begin{pmatrix} u&u'&u'''\end{pmatrix} + 3\frac{\pr{u'}{u''}}{u_3\|u'\|_J^3}\det\begin{pmatrix} u&u'&u''\end{pmatrix}.
\]
The invariants $k_0$, $k_1$ and $k_2$ and their derivatives functionally generate any other invariant for curves in the light cone under the centro-affine action of the Lorentzian group.
\end{theorem}
\begin{proof}
Since $K_L$ is the right Maurer-Cartan matrix, we can make use of theorem \ref{Knorm} and try to find $K_L$ solving the equations 
\begin{equation}
\label{MB}
K_L\cdot I_r = -I_{r+1}+(I_{r})_x
 \end{equation}
where $I_k = c_k$ when normalized, otherwise $I_k = \rho u_k$. The dot represents the infinitesimal prolonged action on derivatives, which, since the action is linear, is given by $K\cdot c_r = Kc_r$.  

When $r=0$ we have
\[
K_L e_4 = -c_1
\]
which determines the last column of $K$ as shown in the statement of the theorem.

When $r = 1$
\[
\|u'\|_J K_L e_2 = K_L c_1 = - c_2 + (c_1)_x = -\begin{pmatrix}a\\ b\\ c\\ d\end{pmatrix}  + \begin{pmatrix} 0\\ (\|u'\|_J)_x\\0\\ d\end{pmatrix} = -\begin{pmatrix} a\\ 0\\0\\ d\end{pmatrix}
\]
which also determined the second column of $K$ to be as shown in the statement. 

Since $K_L\in \g,$ we now have all entries except for the ones in place $(3,1)$, equal to the one in place $(4,3).$ We have called this entry $-k_2$ in the statement. To proceed, we need to use the value $r = 2$ in the recurrence relation for $K$, and, in particular, we will need to find the third entry of $c_3,$ (call it ${c}^3_3$). This entry is of particular interest since $$K\cdot c_2 = -c_3 + (c_2)_x$$ implies $$k_{2}\|u'\|_J = -{c}^3_3,$$ an equation that will allow us to solve for $k_2$ once we know ${c}^3_3$.

Since we know $\rho_L$, calculating $\rho_L u''' = c_3$ directly we conclude 
$$
{c}^3_3 = 3\pr{u''}{u'} \xi_2 + e_2^T A \left(\hat u'''-u_3'''\frac{\hat u}{u_3}\right) 
$$
$$= \frac1{u_3\|u'\|_J}\det\begin{pmatrix} u&u'&u'''\end{pmatrix} + 3\frac{\pr{u'}{u''}}{u_3\|u'\|_J^2}\det\begin{pmatrix} u&u'&u''\end{pmatrix}.$$ This gives the formula for $K_L$ as stated in the Theorem.
\end{proof}

\subsection{Invariant flows of non-radial curves}
Once we have a moving frame we can write down a formula for the most general flow of curves invariant under the action of the group, according to Theorem \ref{invev}. Let $\Phi_g: C_L \to C_L$ be the map given by the action of $G$ on our manifold, that is $\Phi_g(x) = g\cdot x$, where this action is given as in  (\ref{linear}). If $g$ is given by (notice that this is the inverse of (\ref{decomposed}), reflecting the fact that theorem \ref{invev} requires a left moving frame - the inverse of a right one)
\begin{equation}\label{gleft}
g = \begin{pmatrix}1&-v^T&\frac{1}{2}\|v\|^2\\0&I&-v\\0&0&1\end{pmatrix}
\begin{pmatrix}\alpha^{-1}&0&0\\0&A^{-1}&0\\0&0&\alpha\end{pmatrix}
\begin{pmatrix}1&0&0\\-\xi&I&0\\\frac{1}{2}\|\xi\|^2&-\xi^T&1\end{pmatrix} = g_{-1} g_0 g_1
\end{equation}
then $d\Phi_g(o)$ is given by $d\Phi_{g_{-1}}(o)d\Phi_{g_0}(o)d\Phi_{g_1}(o)$ (recall that $o = e_4$). One can easily calculate each of these factors as
\[
d\Phi_{g_{-1}}(o) \begin{pmatrix} \hat w\\ w_3\end{pmatrix} = \begin{pmatrix} I&-v\\ 0&1\end{pmatrix}\begin{pmatrix} \hat w\\ w_3\end{pmatrix},\hskip 2ex d\Phi_{g_0}(o)\begin{pmatrix} \hat w\\ w_3\end{pmatrix} = \begin{pmatrix} A^{-1}&0\\ 0 & \alpha\end{pmatrix} \begin{pmatrix} \hat w\\ w_3\end{pmatrix}
\]
\[
d\Phi_{g_1}(o)\begin{pmatrix} \hat w\\ w_3\end{pmatrix} = \begin{pmatrix} I&0\\ -\xi^T&1\end{pmatrix} \begin{pmatrix}\hat w\\ w_3\end{pmatrix}.
\]
Therefore, if $\rho_L^{-1}$ is the left moving frame associated to the right moving frame previously found,
\[
d\Phi_{\rho_L^{-1}}(o) = \begin{pmatrix} A^{-1}-\hat u\xi^T & \hat u\\ -u_3\xi^T & u_3\end{pmatrix}
\]
Using Theorem \ref{invev}, the following result is immediate.
\begin{theorem} Assume $\bu = \begin{pmatrix} \hu\\ u_3\end{pmatrix} = \begin{pmatrix} u_1\\ u_2\\ u_3\end{pmatrix} $
\[
\bu_t = F(\bu, \bu', \bu'', \dots) 
\]
is an evolution of curves in the cone, invariant under the Lorentzian action; that is, $O(3,1)$ takes solutions to solutions. Then, there exist $r_1, r_2, r_3$ functions of the invariants $k_0, k_1, k_2$ and their derivatives such that
\begin{equation}\label{uevfor}
\bu_t = \begin{pmatrix} A^{-1}-\hu\xi^T & \hu\\ -u_3\xi^T & u_3\end{pmatrix}\begin{pmatrix} r_1\\ r_2\\ r_3\end{pmatrix}
\end{equation}
where $A$ and $\xi$ are as in (\ref{A}),  (\ref{xi1}) and (\ref{xi2}). 
\end{theorem}
Using this theorem we could, in principle, find directly the invariant evolutions that preserve the arc-length $k_0 = \|\bu'\|_J$. But direct calculations are quite involved, so we will choose a simpler path using theorem \ref{Nsplit}. Assume $u$ evolves as in (\ref{uevfor}) and assume that $N = -(\rho_L)_t\rho_L^{-1}$ is the evolution of the left moving frame $\rho_L^{-1}$. Now, consider the section
\[
\s (u) = \begin{pmatrix} u_3^{-1} &\frac{u^T}{u_3}&\frac1{2 u_3} u\cdot u\\ 0&I&u\\ 0&0&u_3\end{pmatrix}
\]
so that $d\s(o) v$ is given by
\[
\begin{pmatrix} -v_3 & v^T & 0\\ 0&I&v\\ 0&0&v_3\end{pmatrix}
\]
Using Theorem \ref{Nsplit} we can conclude that $N_m$, the component of $N$ in the direction of the section given by the $g_{-1}$ factor, is given by $d\s(o)\begin{pmatrix}r\\ r_3\end{pmatrix}$. In this case we have 
\[
N_m = \begin{pmatrix} r_3 & r^T& 0 \\ 0&0&r\\ 0&0& r_3\end{pmatrix}.
\]
On the other hand, if $K$ is the left Maurer-Cartan matrix, the commutation  between $\frac d{dx}$ and $\frac d{dt}$ (or the structure equations) state that
\[
K_t = N_x + [K,N].
\]
Assume that 
\[
N = \begin{pmatrix} -r_3 & r^T &0\\ n & N_0& r\\ 0&n^T & r_3\end{pmatrix}.
\]
Using the fact that  $K = -K_L$ where $K_L$ is as in (\ref{KL}), we can find 
\[
(k_0)_t = r_1'+ r_3 k_0.
\]
Therefore, the condition on our invariant evolution to preserve arc-length is given by
\begin{equation}\label{alpre}
r_3 = -\frac 1{k_0} r_1'.
\end{equation}
The structure equations also give you the evolution of the invariants $k_1, k_2$. Indeed, in $n = (n_1, n_2)^T$ and $N_0 = \begin{pmatrix} o&n_0\\ -n_0&0\end{pmatrix}$, the structure equations become
\[
\begin{pmatrix} 0& k_0& 0 & 0 \\ k_1&0&0& k_0\\ k_2 & 0&0&0\\ 0&k_1&k_2&0\end{pmatrix}_t 
\]
\[
=  \begin{pmatrix} -r_3 & r^T &0\\ n & N_0& r\\ 0&n^T & r_3\end{pmatrix}_x + \begin{pmatrix} k_0n_1-k_1r_1-k_2r_2&k_0r_3&k_0 n_0&0\\ -k_1r_3-k_2n_0&0&k_1r_2-k_2r_1+k_0n_2&\ast\\ -k_2r_3+k_1n_0&\ast&0&\ast\\ 0&\ast&\ast&\ast\end{pmatrix}
\]
where $\ast$ indicate entries that are determined by the group from the entries shown. These equations imply
\[
n_0 = -\frac 1{k_0} r_2', \hskip 2ex n_1 = \frac 1{k_0}(k_1 r_1+ k_2 r_2 + r_3'), \hskip 2ex n_2 = \frac 1{k_0} \left(\left(\frac 1{k_0}r_2'\right)'+k_2r_1-k_1r_2\right)
\]
and they give us also the evolutions of $k_1$ and $k_2$. These are given by
\begin{equation}\label{kev}
\begin{array}{ccc} (k_1)_t &=& \left(\frac1{k_0}(k_1r_1+k_2r_2+ r_3')\right)'-k_1r_3+\frac{k_2}{k_0} r_2'\hskip 6ex\\ (k_2)_t &=& \left(\frac1{k_0}\left(\frac1{k_0}r_2'\right)'+k_2r_1-k_1r_2\right)'-k_2r_3-\frac{k_1}{k_0}r_2'\end{array}
\end{equation}

Finally, notice that, if we assume $k_0 = 1$ and the evolution preserves this arc-length, then $r_3 = - r_1'$ and we get the equations
\begin{equation}\label{P}
\begin{pmatrix} k_1\\ k_2\end{pmatrix}_t = \begin{pmatrix} -D^3+k_1D+Dk_1&Dk_2+k_2D\\ Dk_2+k_2D& D^3-Dk_1-k_1D\end{pmatrix} \begin{pmatrix} r_1\\ r_2\end{pmatrix} = P\begin{pmatrix} r_1\\ r_2\end{pmatrix}.
\end{equation}

The operator $P$ is known to be one of the Hamiltonian structures for a system of complexly coupled KdV equations. 
\section{The local geometry of curves on the $2$-M\"obius sphere}

In our next section, we will do a study of flows in the conformal sphere, parallel to the one we just did for the cone. The results were previously found in (\cite{M2}) for the general case $O(n+1,1)$, but we present them again readjusting the calculations to match those of the cone.

\subsection{Group based moving frame} In (\ref{sphereact}) we described the action of $O(3,1)$ on $M$. As before we will use it to find a right moving frame through recurrent prolongations. Assume $g$ is as in (\ref{decomposed}) with slightly different notation ($B$ instead of $A$, $w$ instead of $v$, $\eta$ instead of $\xi$ and $\beta$ instead of $\alpha$). 

\vskip 2ex
{\it Zeroth order normalization}

As in the light cone case, we first normalize $g_3\cdot m = {\bf 0} = c_0$.  We can clearly solve this equation by choosing $w = -m.$  This suffices since $g_1$ and $g_0$ are stabilizers of $o = {\bf 0}$ under the conformal action.

{\em First order normalization}

Continuing to the first order normalization, we need to find the first prolongation by differentiating the action (\ref{sphereact}) and then substitute $w = -m$.  The result is 
\[
g\cdot m' = \beta B m'
\]
If we choose $c_1 = e_1$, then   
\begin{equation}\label{alphac}
\beta^{-1} = \|m'\|
\end{equation}
and 
\begin{equation}\label{Ac}
B = \frac{1}{\|m'\|}\begin{pmatrix}m_1'&m_2'\\-m_2'&m_1'\end{pmatrix}.
\end{equation}

{\em Second order normalization}

We now differentiate the action again, substitute the already found values, and normalized to $c_2 = \begin{pmatrix}0\\0\\\end{pmatrix} = {\bf 0}$. 

The second prolonged action then becomes
\[
g\cdot m'' = \beta(\beta\eta \|m'\|^2 + B m'') + 2\eta_1 e_1
\]
which, when made equal to ${\bf 0}$ allows us to solve for $\eta$ 
\[
\eta_1 = \frac{m'\cdot m''}{\|m'\|^2}, \hskip 2ex \eta_2 = -\frac{1}{\|m'\|^2}\det \begin{pmatrix}m_1'&m_1''\\m_2'&m_2''\\\end{pmatrix}.
\]

We now have the explicit form of the moving frame
\begin{equation}
\label{conformalframe}
\rho_M = \begin{pmatrix}1&-m^T&\frac{1}{2}\|m\|^2\\\eta&B-\eta m^T&\frac{1}{2}\eta\|m\|^2-Bm\\\frac{1}{2}\eta_2 ^2&-\frac{1}{2}\eta_2 ^2m^T+\eta^TB&\frac{1}{4}\eta_2 ^2\|m\|^2 - \eta^TBm + 1\end{pmatrix},
\end{equation}
where $B$ and $\eta$ are given above.

\subsection{Maurer Cartan matrix and differential invariants}

As in the cone case, we can find the Maurer Cartan matrix associated to this moving frame and with it a generating system of differential conformal invariants.

\begin{theorem} If $K_M = (\rho_M)_x\rho_M^{-1}$, then
\begin{equation}\label{KM}K_M = \begin{pmatrix}0&-1&0&0\\ -\kk_1 &0&0&-1\\
-\kk_{2}&0&0&0\\0&-\kk_1&-\kk_2&0\end{pmatrix}
\end{equation} 
where 
\[
 \kk_1 = \frac{m'\cdot m'''}{\|m'\|^2} - \frac32\frac{(m'\cdot m'')^2}{\|m'\|^4}+\frac32 \frac{\det^2(m'\hskip1ex m'')}{\|m'\|^4}
\]
and
\[
\kk_2 = \frac{\det(m'\hskip1ex m''')}{\|m'\|^2}-3\frac{(m'\cdot m'')\det(m' \hskip1ex m'')}{\|m'\|^4}.
\]
The invariants $\kk_1$, $\kk_2$ and their derivatives functionally generate any other invariant for non-degenerate curves in the light cone.
\end{theorem}
\begin{proof}
As before we can use the relation
\[
K_M\cdot I_r = -I_{r+1}+(I_r)_x
\]
to find $K_M$, where the action $\cdot$ indicates the infinitesimal prolonged action associated to (\ref{sphereact}). For $r=0$ this action is given by
\[
K_M\cdot c_0 = \frac12 K_\eta\|c_0\|^2 + K_Bc_0+K_w - c_0(-b+K_\eta^Tc_0) = -c_1+(c_0)_x
\]
where 
\[
K_M = \begin{pmatrix} b & K_w^T & 0\\ K_{\eta} & K_B & K_w\\ 0 & K_\eta^T & -b\end{pmatrix}
\]
with $K_B$ skew symmetric. Substituting $c_0 = 0$ and $c_1 = e_1$ we get $K_w = -e_1$.

Calculating the infinitesimal version of the first prolonged action results in the formula for $r = 1$, that is
\[
K_M \cdot c_1 = K_\eta c_0^Tc_1 + K_B c_1 -c_1(-b+K_\eta^Tc_0)-c_0K_\eta^Tc_1= -c_2 + (c_1)_x
\]
which after substituting $c_0 = c_2 = 0$ and $c_1 = e_1$ becomes $K_Be_1 = 0$, and $b=0$ that is $K_B = 0, b=0$ since $K_B$ is skew-symmetric.

Finally we use the case $r=2$. We need to calculate the infinitesimal version of the second prolonged action. It is given by
\[
K_M\cdot c_2 = K_\eta(c_1^Tc_1+c_0^Tc_2)+K_Bc_2-c_2(K_\eta^Tc_0-b)-2c_1K_\eta^Tc_1-c_0K_\eta^Tc_2 =   -c_3 + (c_2)_x 
\]
which becomes $K_\eta - 2K_\eta = -K_{\eta} = -c_3$ after normalization. Therefore $K_\eta = c_3 = \rho\cdot m'''$. We can calculate directly this product as before. If $\rho_M\cdot m''' = \frac EF$, a fraction given as in in (\ref{sphereact}), then using previous normalizations we obtain
\[
\rho_M\cdot m''' = \frac{E'''}F-3\frac{F''}{F}e_1.
\]
After long but straightforward calculations we obtain that 
\[
c_3 = \rho_M\cdot m''' = \begin{pmatrix}\kk_1\\\kk_2\end{pmatrix}
\]
where $\kk_1$ and $\kk_2$ are as given in the statement of the theorem. \end{proof}
\subsection{Invariant evolutions of conformal curves}

As in the case of the cone, we would like to find the formula for any invariant evolution of curves in the conformal sphere associated to the left moving frame $\rho^{-1}$ as in theorem \ref{invev}. The action this time is given by (\ref{sphereact}), and, as before, if $g = g_{-1}g_0g_1$ is as in (\ref{gleft}), then $d\Phi_g(o) = d\Phi_{g_{-1}}(o)d\Phi_{g_0}(o)d\Phi_{g_1}(o)$. In this case is is actually simpler. From the formula one can see that $g_{-1}$ acts by translation, and so $ d\Phi_{g_{-1}}(o) = I$. Also, 
\[
\Phi_{g_1}(m) = g_1\cdot m = \frac{m-\frac12\eta \|m\|^2}{1-\eta^Tm+\frac14 \|\eta\|^2\|m\|^2}
\]
and one can check directly that $d\Phi_{g_1}(o) = I$. Therefore, $d\Phi_g(o) = d\Phi_{g_0}(o)$. Finally, $\Phi_{g_0}(m) = \beta B m$, therefore $d\Phi_{\rho}(o) = \beta^{-1} B^{-1}$, where $\beta$ and $B$ are given as in (\ref{Ac}) and (\ref{alphac}). We just prove the following result.

\begin{theorem} Assume $m$ is a solution of a conformally invariant evolution of the form
\[
m_t = F(m, m', m'', \dots).
\]
Then
\begin{equation}\label{mev}
m_t = \begin{pmatrix} m_1' & -m_2'\\ m_2'&m_1'\end{pmatrix} \begin{pmatrix} s_1\\ s_2\end{pmatrix}
\end{equation} 
where $s_1$ and $s_2$ are functions of the differential invariants $\kappa_1$, $\kappa_2$ and their derivatives.
\end{theorem}

As before, we can find the evolution that $m_t$ induces on the invariants $\kappa_i$, using theorem \ref{Nsplit} and the structure equations.

Consider the section of the sphere
\[
\s(m) = \begin{pmatrix} 1&m^T & \frac12 \|m\|^2\\ 0& I&m\\ 0&0& 1\end{pmatrix}
\]
so that
\[
d\s(o)v = \begin{pmatrix} 0&v^T&0\\ 0&0&v\\ 0&0&0\end{pmatrix}.
\]
According to Theorem \ref{Nsplit}, if $m(t,x)$ satisfies equation (\ref{mev}), and it $\rho_M(t,x)$ is the flow of right moving frames associated to $m$, then
\[
N = -(\rho_M)_t\rho_M^{-1} = \begin{pmatrix} \alpha & \bs^T & 0\\ n&N_0&\bs\\ 0&n^T&-\alpha\end{pmatrix}
\]
where $\bs = (s_1, s_2)^T$. Furthermore, if $K = -K_M$ is the left Maurer-Cartan matrix, the structure equations
\[
K_t = N_x + [K,N]
\]
will be identical to those of the cone, with $k_0 = 1$, $\bs = (s_1, s_2)^T$ taking the role of $(r_1, r_2)^T$ and $-\alpha$ substituting $r_3$. That is, the evolution of $\kappa_i$ will coincide with that of $k_i$ if $k_0 = 1$ with $s_i$ instead of $r_i$. The equation also forces the value $\alpha = s_1'$. These conclusions prove the following theorem. 

\begin{theorem} If $m(t,x)$ satisfies the equation (\ref{mev}), then its invariants $\kappa_1$, $\kappa_2$ satisfy the equation
\[
\begin{pmatrix} \kappa_1\\ \kappa_2\end{pmatrix}_t = P \begin{pmatrix} s_1\\ s_2\end{pmatrix}
\]
where $P$ is the operator given in (\ref{P}).
\end{theorem}
\section{The relationship between the two geometries}

As we just saw, the evolution induced on $\kappa_i$ by the invariant evolution $m_t$ coincides with that of $k_i$ induced by $u_t$ whenever $r_i = s_i$, $i = 1,2$, and whenever $u$ is parametrized by arc-length. In this section we will show that, in fact, $k_1$ are taken to $\kappa_i$ under projectivization and the curve evolutions themselves $u_t$ and $m_t$ are also related by projectivization.                   

\subsection{Differential invariants}

We will show that the two remaining invariants of the cone and the sphere are 1-to-1 related if we re-parameterize to have $\|u'\|_J = 1$. Let $\tilde m$ be given by
\begin{equation}
 \tilde{m} = \begin{pmatrix}\frac{1}{2}\|m\|^2\\m_1\\m_2\\1\end{pmatrix}
 \end{equation}
 and let
 \(
 {u} = \tilde{m}u_3, m = (m_1,m_2)^T\) be an associated lift and its coordinates.
 Notice that $\|\tilde m'\|_J = \|m'\|$, $\pr{\tilde m'}{\tilde m''} = \langle{ m'},{ m''}\rangle$, and likewise for higher order derivatives. Notice also that, if $\|u'\|_J = 1$, then
\[
1 = {u}'^TJ{u}' = (u_3\tilde{m})'^TJ(u_3\tilde{m}) = u_3^2\|\tilde{m}'\|_J^2 + 2u_3u_3'\pr{\tilde{m}}{\tilde{m}'}+u_3'^2\|\tilde{m}\|_J^2
\]
\[= u_3^2\|\tilde m'\|_J^2 .
\]
Therefore $u_3 = \frac{1}{\|\tilde{m}'\|_J}$ uniquely ensures that the lift of $m$ to the cone  defined by $u$ is parametrized by arc-length. Let us call that lift $\Lambda: M \to C_L$, $\Lambda(m) = u = \frac{\tilde{m}}{\|\tilde m'\|_J}$.

Next, 
$$
\pr{u'}{u'''} = 3u_3''u_3\|m'\|^2+3u_3'u_3'\pr{\tilde{m}}{\tilde{m}''}+u_3^2\langle{m'},{m'''}\rangle.
$$  
So 
$$
k_1 = \frac{\pr{u'}{u'''}}{2} = -\frac{3}{4}\frac{(\|m'\|^2)_x^2}{\|m'\|^4} + \frac{\pr{m'}{m'''}}{\|m'\|^2} + \frac{3}{2}\frac{\|m''\|^2}{\|m'\|^2} = \kk_1
$$ 
where we have used the relation 
$$\pr{m'}{m''}^2+\det(m"\hskip1ex m'')^2 = \|m''\|^2\|m'\|^2.
$$ 

Equally
\[
k_2 = u_3^2\det(m'\hskip1ex m'') + 3 u_3u_3'\det(m'\hskip1ex m'') 
\]
\[
= \frac{\det(m'\hskip1ex m''')}{\|m'\|^2}-3\frac{m'\cdot m''}{\|m'\|^4}\det(m'\hskip1ex m'') = \kk_2.
\]
\subsection{Invariant evolutions}
We will now relate invariant evolutions in both the cone and sphere as reflected in the following theorem.
\begin{theorem}
Assume $\bu$ evolves following the equation (\ref{uevfor}).
  Assume $r_3 = r_1'$ so that the evolution preserves arc-length, and assume further that our initial condition satisfies $\|u'\|_J = 1$. Let $m = \Pi (u)$ be the projectivization of $u$. Then $m$ satisfies equation (\ref{mev}) with $s_i = r_i$, $i = 1,2$. And vice-versa, a flow $m$ in the sphere can be lifted uniquely to a flow $u$ in the cone parametrized by arc-length and the lifting takes the $m$ evolution to the $\bu$ evolution.
\end{theorem}
\begin{proof}
Assume $\bu$ satisfies (\ref{uevfor}) and let $m = \frac 1{u_3} \hat u$, where, as before, $\bu = \begin{pmatrix} \hat u\\ u_3\end{pmatrix}$. Then
\[
m_ t = \frac 1{u_3}\hat u_t - \frac {(u_3)_t}{u_3^2} \hat u = \frac1{u_3}\left((A^{-1}-\hat u\xi^T)\begin{pmatrix} r_1\\ r_2\end{pmatrix} + r_3 \hat u\right) + \frac 1{u_3^2} \left(u_3\xi^T \begin{pmatrix} r_1\\ r_2\end{pmatrix} - u_3 r_3\right) \hat u
\]
\[
= \frac 1{u_3} A^{-1} \begin{pmatrix} r_1\\ r_2\end{pmatrix}
\]
the matrix $A$ given as in (\ref{A}). Now, from (\ref{A}) we can see that
\[
A^{T}\begin{pmatrix} 1\\ 0\end{pmatrix} = \frac1{\|u'\|_J}\left(\hat u'-\frac{u_3'}{u_3} \hat u\right) = \frac{u_3}{\|u'\|_J}\left(\frac{\hat u}{u_3}\right)' = \frac {u_3}{\|u'\|_J} m'.
\]
Since $A^T = A^{-1}$, this determines $A^{-1}$ completely. Finally, if $u$ is parametrized by arc-length, 
\[
\|u'\|_J^2 = \hat u'\cdot \hat u' - 2 u_3'\left(\frac1{u_3} \hat u\cdot \hat u' - \frac{u_3'}{2u_3^2}\|\hat u\|^2\right) = \|m'\|^2 u_3^2.
\]
Therefore $u_3 = \|m'\|^{-1}$ and the formula for the evolution of $m$ is as in (\ref{mev}). We can clearly walk our way back to obtain the second part of the theorem using the unique lifting $m \to u$ where $u_3 = \|m'\|^{-1}$.

\end{proof}

\subsection{Geometric Hamiltonian structures}

As we explained in section \ref{geomst}, any space of invariants $\K$ associated to a homogeneous space $G/H$ inherits a Poisson structure linked to invariant evolutions of curves in $G/H$ as in theorem \ref{greal}. In this section we will show that the projectivization map $\Pi$ induces a Poisson isomorphism between the space of invariants of curves parametrized by arc-length in the light cone (as given by the section $K_L$ in (\ref{KL}) with $|\!|u'|\!|_J = 1$) and the space of invariants of unparametrized curves in the conformal sphere (as given by the section $K_M$ in (\ref{KM})). Let's denote the former by $\K_L^1$ and the latter by $\K_M$.

\begin{theorem} The projectivization map $\Pi$ induces a Poisson map between the spaces of differential invariants $\K_L^1$ and $\K_M$ endowed with their respective geometric Poisson brackets. Furthermore, $\K_L$ with its geometric Poisson bracket foliates into Poisson submanifolds corresponding to constant values of $k_0$ and the Poisson sunmanifold for $k_0 = 1$ is Poisson equivalent to $\K_M$.
\end{theorem}
\begin{proof}

Notice that both geometric Poisson brackets are reductions of the same bracket, namely (\ref{br1}), to the quotients $\hat U/\Lo H_L$ and $\tilde U/\Lo H_M$ where $\hat U$ and $\tilde U$ are open subsets of $\Lo\oa(3,1)^\ast$ (since these are local results, we will define the manifolds as given by the non-void intersection $U = \hat U\cap \tilde U$. See \cite{M1}). Notice also that $H_L \subset H_M$, in fact $H_L$ is a subgroup of $H_M$. Therefore, there is a natural inclusion 
\[
\K_M \equiv U/\Lo H_M \subset  U/\Lo H_L \equiv\K_L.
\]
This inclusion is easily described by $\K_M$ as an affine subspace of $\K_L$ (given by $k_0 = 1$). 

Now, from theorem \ref{Kred} we know that (\ref{br1}) reduces to both $\K_L$ and $\K_M$, and hence we can conclude that the geometric bracket in $\K_M$ is a further reduction of the bracket in $\K_L$ to the subspace $k_0 = 1$.

In fact this further reduction is simply a restriction as $\K_M$ is a Poisson submanifold of $\K_L$. This can be easily seen by finding it   explicitly.
         
The cone is isomorphic to $O(3,1)/H_L$, where $H_L$ in this case is the isotropy group of $e_4$ under the linear action.  That is, it corresponds to $\alpha = 1$, $v = 0$ in (\ref{decomposed}). We notice that the algebra $\h_L$ is given by
$z = 0$, $a = 0$ in (\ref{alg}). Therefore $\h_L^0$ is given, under identification with its dual using the trace, by matrices of the form
\[
\begin{pmatrix} \ast&0&0\\ \ast & 0 & 0\\ 0&\ast&\ast\end{pmatrix}.
\]

To calculate the reduction of the bracket (\ref{br1}) we will need to use constant extensions to the gauge leaves of functionals defined on $\K_L$. From the form of $K_L$ in (\ref{KL}), any such an extension will need to have a variational derivative of the form
\begin{equation}\label{HL}
\frac{\delta \HH_L}{\delta L}(K) = \begin{pmatrix}  c&\frac12 h_1&\frac12 h_2&0\\ \frac12 h_0&0&d&\frac12 h_1\\ e & -d&0&\frac12 h_2\\ 0&\frac12 h_0&e&-c\end{pmatrix}
\end{equation}
where $h_i = \frac{\delta h}{\delta k_i}(\bok)$.
If the extension is constant on the gauge leaves, infinitesimally we will have
\[
\left(\frac{\delta \HH_L}{\delta L}(K)\right)_x + \left[K, \frac{\delta \HH_L}{\delta L}(K)\right]\in \h_L^0
\]
where $K = -K_L$ since we are considering left Maurer-Cartan matrices. This condition splits into the following equations for the entries of $\frac{\delta \HH_L}{\delta L}(K)$, 
\[
\frac12 h_1' - k_0c = 0, \hskip 2ex \frac12 h_2' + k_0d = 0, 
\]
\[
 k_0e +d' -\frac12(k_2h_1-k_1h_2)=0
 \]
and they allow us to solve for the entries of $\frac{\delta \HH_L}{\delta L}(K)$
 \[
c = \frac 1{2k_0} h_1', \hskip 2ex d = -\frac1{2k_0} h_2', \hskip 2ex e = \frac1{k_0}\left(\frac1{2k_0} h_2'\right)'+\frac1{2k_0} (k_2h_1-k_1h_2).
 \]
 
 If two such an extensions are found, we merely need to use (\ref{br1}) to find our reduced bracket. That is, let $h$ and $f$ be two functionals on $\K_L$ and let $\HH_L$ and $\F_L$ be two extensions found as above. Then
 \[
 \{h, f\}(\bok) = \int_{S^1} \mathrm{trace}\left(\frac{\delta \F_L}{\delta L}(K)\left(\left(\frac{\delta \HH_L}{\delta L}(K)\right)_x + \left[K, \frac{\delta \HH_L}{\delta L}(K)\right]\right)\right) dx
 \]
 
\[
= \int_{S^1}\left(f_1(\frac12 h_0'+k_1c-k_2d)+\frac1{k_0}f_1'(c'+\frac12 k_0h_0 - \frac12 (k_1h_1+k_2h_2)) + f_2(e'+k_2c+k_1d)\right)dx
\]
\[
= \frac12 \int_{S^1} \begin{pmatrix}f_0& f_1&f_2\end{pmatrix} \begin{pmatrix}0&0&0\\ 0&-D\frac1{k_0} D\frac1{k_0}D+D\frac{k_1}{k_0}+\frac{k_1}{k_0}D& \frac{k_2}{k_0} D + D\frac{k_2}{k_0}\\ 0&\frac{k_2}{k_0} D + D\frac{k_2}{k_0} & D\frac1{k_0}D\frac1{k_0}D - D\frac{k_1}{k_0}-\frac{k_1}{k_0}D\end{pmatrix}\begin{pmatrix}h_0\\ h_1\\ h_2\end{pmatrix} dx.
\]
 \end{proof}
A direct consequence of this theorem is the following corollary.
\begin{corollary}
If an invariant evolution of curves in the cone induces a Hamiltonian evolution on its invariants, the evolution needs to preserve arc-length. \end{corollary}

This corollary is somehow surprising. Although it is true that the invariant projective subspace is a Poisson submanifold of the cone invariants, there was no a-priori reason why the cone should not have its own evolutions inducing Hamiltonian evolutions on its invariants and not necessarily preserving arc-length. After all, invariants in the cone constitute a Poisson manifold on their own independently from the sphere. What this calculation tells us is that it is, in fact, the invariant conformal Poisson manifold in disguise and geometric realizations of Hamiltonian systems that do not preserve arc-length do not exist.

\section{A completely integrable flow by non-radial curves on the light cone}

The conformal sphere hosts conformal geometric realizations of a system of complexly coupled KdV equations. This was known originally in \cite{M3}, and it was proved later in \cite{M2} for any dimension. The relation between cone and conformal sphere allows us to translate this realization to the cone and to obtain an geometric realization of a complexly coupled KdV system by star-shaped flows in the cone. The following theorem follows directly from our calculations.

\begin{theorem}
  Consider the Hamiltonian functional 
\[
h(\bok) = \frac12\int_{S^1} (k_1^2+k_2^2) dx.
\]
Then, the star-shaped flow of curves in the cone given by the evolution
\begin{equation}\label{geoKdV}
\bu_t = \begin{pmatrix} A^{-1}-\hu\xi^T & \hu\\ -u_3\xi^T & u_3\end{pmatrix}\begin{pmatrix} k_1\\ k_2\\ k_1'\end{pmatrix} 
\end{equation}
preserves arc-length and, if the initial condition is parametrized by arc-length, then it induces a complexly coupled system of KdV equations on the invariants $k_1, k_2$.
\end{theorem}

\begin{proof}
Notice that, according to (\ref{kev}), and since $k_0 = 1$, the evolution   on $k_1, k_2$ induced by this curve evolution is 
\[
\begin{array}{ccc}(k_1)_t &=& -k_1'''+3 k_1k_1' + 3 k_2k_2'\\
(k_2)_t &=& k_2''' +k_1'k_2 - k_1k_2'
\end{array}
\]
that is, a complexly coupled KdV system.
\end{proof}

This system of equations is known to be completely integrable and biHamiltonian; that is, Hamiltonian with respect to two different but compatible Poisson structures (compatible means that their sum is also Poisson). Both brackets (\ref{br1}) and (\ref{br2}) are Poisson and compatible for any choice of $L_0$. To finish the paper we will show that, indeed, (\ref{br2}) is also reducible to both $\K_L$ and $\K_M$ (the $\K_M$ case was already proved in \cite{M2}), its reduction also foliates according to different values of $k_0$, and when $k_0 = 1$ both cases are identical. Notice that there are examples (\cite{M4}) when (\ref{br2}) is {\it not} reducible to $\K$: when $G = \SL(4,\R)$ and $H$ is chosen so that $G/H$ is the Lagrangian Grassmannian, (\ref{br2}) does not reduce for {\it any} value of $L_0$. Only the reduction of (\ref{br1}) is guaranteed.
\begin{theorem} The Poisson bracket (\ref{br2}) is reducible to both $\K_L$ and $\K_M$. The reduced bracket on $\K_L$ foliates into Poisson submanifolds corresponding to different constant values of $k_0$, and when $k_0 = 1$ the Poisson submanifold $\K_L^1$ is equivalent to $\K_M$.
\end{theorem}  
\begin{proof}
To calculate the possible reduction we need to proceed as in the case of the main bracket. If $h:\K_L \to \R$, and if $\HH_L$ is an extension of $h$, constant on the leaves of $\Lo H_L$, then $\frac{\delta \HH_L}{\delta L}(K_L)$ is determined to be as in (\ref{HL}). If $\F_L$ is a similar extension of a Hamiltonian functional $f$, the possible reduction of (\ref{br2}) would then be given by
\[
\{h,f\}_0(\bok) = \int_{S^1}\mathrm{trace}\left(\frac{\delta \F_L}{\delta L}(K_L)\left[L_0, \frac{\delta \HH_L}{\delta L}(K_L)\right]\right)dx.
\]
Let us choose 
\[
L_0 = \begin{pmatrix} 0& 0&0\\ e_1&0&0\\ 0&e_1^T&0\end{pmatrix}.
\]
With this choice, the possible reduction looks like (we are denoting by $d_h, d_f$ the entries in (\ref{HL}) corresponding to the extensions of $h$ and $f$, respectively. Likewise for the other entries.)
\[
\{h, f\})(\bok) = \int_{S^1} (h_1 c_f- f_1 c_h+h_2 d_f- f_2 d_h) dx =\int_{S^1} \frac1{2k_0}(h_1 f_1'-f_1h_1'-h_2 f_2'+f_2h_2') dx
\]
\[
= \frac12\int_{S^1} \begin{pmatrix} f_0&f_1&f_2\end{pmatrix} P_0 \begin{pmatrix} h_0\\ h_1\\ h_2\end{pmatrix}
\]
where
\[
P_0 = \begin{pmatrix} 0&0&0&\\ 0& -D\frac 1{k_0} - \frac1{k_0} D&0\\ 0&0& D\frac 1{k_0} + \frac1{k_0} D \end{pmatrix}.
\]
It is a very simple exercise to prove that $P_0$ is a Poisson tensor, it preserves constant values of $k_0$, and when $k_0 = 1$, the bracket restricts to 
\[
\begin{pmatrix} -D&0\\ 0&D\end{pmatrix}.
\]
This is the second Hamiltonian structure for the complexly coupled KdV system.
\end{proof}

Notice that if we have a solution of the complexly coupled system - including soliton solutions - we can directly find a solution for the curve evolution (\ref{geoKdV}) using the moving frame $\rho_L$ along $u(t,x)$. Indeed, if $k_i(t,x)$ is such a solution we can re-construct $\rho_L$. The frame $\rho_L$ will be the solution of the linear system $(\rho_L)_x = \rho_L K_L$ with appropriate initial conditions determined by the initial conditions of the KdV solution. Once we solve this linear system we can readily find a solution of (\ref{geoKdV}) through the relation $\rho_L\cdot e_4 = u$. This process allows us to describe soliton flows in the cone by simply solving a first order linear system, instead of a system of two third order ODEs.

\end{document}